\documentclass[authoryear,preprint,review,11pt]{elsarticle}

\usepackage{amssymb}
\usepackage{amsmath}

\usepackage[a4paper, total={6in, 8in}]{geometry}
\usepackage{amsthm}
\newtheorem{proposition}{Proposition}
\newtheorem{theorem}{Theorem}
\newtheorem{lemma}{Lemma}
\usepackage{xcolor}
\newcommand{\sylvain}[1]{{\color{black}#1}}

\journal{European Journal of Operations Research}

\begin{document}

\begin{frontmatter}


\title{Approximate Stability of Subadditive Games and Traveling Salesman Games} 


\author[1,2]{Nicolas Besson Niebles \corref{cor1}} 
\ead{nicolas.besson-niebles@grenoble-inp.fr}

\author[2]{Sylvain Bouveret}
\ead{sylvain.bouveret@grenoble-inp.fr}

\author[1]{Nadia Brauner}
\ead{nadia.brauner@grenoble-inp.fr}

\author[3]{Nicolas Brulard} 
\ead{nbrulard@holiag.fr}
\affiliation[1]{organization = {Univ. Grenoble Alpes, CNRS, Grenoble INP, GSCOP},
city = {Grenoble}, 
postcode = {38000}, 
country = {France}
}
\affiliation[2]{organization = {Univ. Grenoble Alpes, CNRS, Grenoble INP, LIG},
city = {Grenoble}, 
postcode = {38000}, 
country = {France}
}

\affiliation[3]{
organization = {Holiag},
city = {Neuville-sur-Saône},
postcode = {69250}, 
country = {France} 
}

\begin{abstract}
    The core of Transferable Utility (T.U.) games is a well-known solution concept from cooperative game theory yielding a cost allocation among $n$ agents (called players) forming a coalition that is stable (i.e. no subset of players has an interest to deviate). In this paper, inspired by a practical application in the context of a decision support system for collaborative transportation in a Short Food Supply Chain (SFSC), we mainly focus on Traveling Salesman Games (TSGs), where the objective is to allocate the cost of a Traveling Salesman Problem (TSP) with $n$ locations and 1 depot to $n$ players, each linked to exactly one of the locations. Given the computational complexity of computing an element of the core and the cost of a TSP, we study \textit{semicore allocations}: a relaxation of the core that only requires that the subsets of size $n - 1$ and of size $1$ do not wish to deviate from the coalition. In the literature, instances of TSGs with empty cores and semicores are found. Hence, this paper first surveys the methods to approximate stability whenever the core is empty, such as the cost of stability (computing the minimum amount of money to subsidize the coalition with to attain stability) and the $\varepsilon$-core (which is a set of allocations that allow subsets of players to exceed their actual cost, but at most of a value of $\varepsilon$). We prove that these two solution concepts are correlated. We also define the cost of semicore stability and the $\varepsilon$-semicore, inspired by the latter concepts. We provide exact formulas for the cost of semicore stability and for the smallest $\varepsilon$ such that the $\varepsilon$-semicore is not empty, for general subadditive games. Finally, we identify a method to bound the cost of stability in the literature, generalize it and translate it to the semicore. The bounds provided for the cost of (semicore) stability may allow the local government to estimate how much they require to subsidize some coalition of farmers with to allow stability.
\end{abstract}



\end{frontmatter}



\section{Cooperative game theory concepts for collaborative transportation in Short Food Supply Chains}\label{sec:intro}

Cooperative game theory provides a theoretical framework for cooperation structures. In particular, it defines Transferable Utility (T.U.) games \citep{shapley}, which represent a way to model many cooperation-related problems, such as coalition formation and value/cost allocation problems. The focus of this paper will be on the latter.

When multiple agents (called \textit{players}) wish to cooperate (or to form a \textit{coalition}) by jointly striving to solve some shared problem, one expects the efforts to be divided in a fair manner among the players. We only focus on the case where the efforts are financial, and thus want to study fair cost allocations. Cooperative game theory provides multiple ways to define what a "fair" cost allocation is. A classical solution concept conveying the idea of fairness is stability. The \textit{core} of a T.U. game is the set of "stable" cost allocations, where no subset of the players has an interest to deviate from their coalition, by paying a value that is smaller than or equal to the cost they would pay if they formed a coalition on their own. However, some specific types of T.U. games may have instances with empty cores. One example is the Traveling Salesman Game (TSG), where the objective is to allocate the cost of a Traveling Salesman Problem (TSP) with $n$ locations and 1 depot to $n$ players, each linked to exactly one of the locations.

The TSG will be the main focus of this paper. The study of this setting is motivated by a practical application to a decision support system for collaborative transportation in an association of farmers in a Short Food Supply Chain, in which the cost of a TSP must be allocated to the farmers. The association is composed of a group of farmers (denoted \sylvain{by} $\mathcal{N}$), sharing a vehicle (that we assume to have unlimited capacity for \sylvain{the sake of simplicity}), and a client (denoted \sylvain{by} $0$), where their production must be brought whenever they receive an order. The tour to pickup the production at the farms and to deliver it to the client is a \sylvain{solution to the} TSP where $\mathcal{N} \cup \{0\}$ is the set of locations to be covered. The transportation cost must be paid by the farmers, which means \sylvain{that} $\mathcal{N}$ is the set of players. 

Notice that when defining the core, we stated that all subsets of the players must have the property of not wanting to deviate. There are $2^n - 1$ non-empty subsets of a set of size $n$; therefore the sole representation of absolute stability requires an exponential number of constraints. Even though there are many efficient methods to compute elements of the core (or to decide whether it is empty) \citep{tsg, corealloctsg}, it may remain unpractical to compute a stable allocation when the number of players increases. The need for less restrictive fairness criteria is even more striking when the computation of the costs to allocate has an exponential time complexity, as is the case of the TSP.

An alternative is to limit the stability constraints to some specific subsets of the players. For instance, the semicore of T.U. games \citep{semicore, semicorenucleolus, semicoretsg} is a set of cost allocations where all players pay less than what they would if they were alone (stability of the coalitions of size $1$) and all players compensate their presence in the coalition, so that the remaining players do not wish to reject them (stability of coalitions of size $n - 1$). Through discussions with a logistics expert from the association studied in our research project, we learned that farmers have expressed their concern with paying more than what they would if they did not belong to the association, or their concerns with adding farmers to the association who are too expensive. This directly correlates with the concept of semicore. One natural question to ask is whether this solution concept, being less restrictive than the core, can be empty in the case of the TSG. It turns out that it can indeed be empty \citep{semicoretsg}, which in turn raises another question: \textit{How to compute allocations that approximate semicore stability whenever the semicore is empty?} 

Many alternatives to the core, whenever it is empty, can be found in the literature. It is possible to search for other solution concepts (such as the Shapley value -- \citealt{shapley}), to subsidize the coalition \citep{costofstability}, or to allow instability but minimizing it as much as possible \citep{originalepsiloncore}. In this paper, we take an interest in the latter two alternatives. The value of the minimum amount of money to subsidize a coalition with, in order to allow stability, is called the \textit{cost of stability}. An allocation allowing some coalitions to pay a value that exceeds their real cost (i.e. the cost that the coalition would pay if they formed a coalition on their own) by at most $\varepsilon$ is called the $\varepsilon$\textit{-core}. In this paper we establish a link between the cost of stability and the $\varepsilon$-core of a game, by proving they are equivalent to some extent. We translate these concepts to the semicore, defining the \textit{cost of semicore stability}, and the $\varepsilon$\textit{-semicore} of T.U. games. We provide an exact formula to compute the cost of semicore stability and the smallest $\varepsilon$ such that the $\varepsilon$-semicore is not empty. We identify in the literature a method to compute bounds for the cost of stability and the optimal $\varepsilon$-core, generalize it, and translate it to the semicore. We apply this method to Traveling Salesman Games. If some local government wanted to provide financial aid to associations of farmers in an SFSC, they may estimate the value that they can subsidize an association with, thanks to these bounds.

The paper is organized as follows. In Section~\ref{sec:def}, we introduce the basic concepts and tools from cooperative game theory that are required to understand the contributions of this paper. In Section~\ref{sec:semicore} we provide further justification to the study of semicore allocations, by giving further detail about the association of farmers of a SFSC inspiring this research. In Section~\ref{sec:approxfair} we introduce various concepts to approximate stability from the literature, and we provide the proof of equivalence between many of these concepts in the case of subadditive games. In Section~\ref{sec:approxsemicore}, we properly define the cost of semicore stability and the $\varepsilon$-semicore of subadditive games, and provide an exact formula for the optimal values of these concepts. In Section~\ref{sec:generalmethodbound}, we generalize a method from the literature to bound the aforementioned values, apply it to the TSG and translate it to the solution concept of the semicore. Finally, in Section~\ref{sec:polybounds} we provide simple, practical and polynomial bounds for the cost of semicore stability.

\subsection{Definitions from cooperative game theory}\label{sec:def}

In this section we introduce the basic notations and concepts from cooperative game theory.

\paragraph*{Transferable Utility Games}

\textit{Transferable Utility} (T.U.) Games, introduced by \sylvain{\cite{shapley}} as $n$-person games, are defined by a pair $(\mathcal{N}, v)$, with $\mathcal{N} = \{1, 2, ..., n\}$ a set of agents (that we call \textit{players}) with a common goal and $v:2^\mathcal{N} \rightarrow \mathbb{R}_+$, a value function such that for a subset $S \subseteq \mathcal{N}$ (called \textit{coalition}), $v(S)$ is the value that could be gained by having the players in $S$ to cooperate. In an analogous manner, we can refer to T.U. games where, instead of willing to maximize some gain, the players want to minimize a cost. We refer to these as \textit{cost function} games. We therefore use the notation $(\mathcal{N}, c)$ for cost function games, where $c(S) \in \mathbb{R}_+$ is the cost of having the players in $S \subseteq \mathcal{N}$ to cooperate. An \textit{outcome} (or \textit{allocation vector}) of a T.U. game is a vector $x = (x_i)_{i \in \mathcal{N}} \in \mathbb{R}_+$, where $x_i$ is the cost that player $i$ must pay and such that $\sum_{i \in \mathcal{N}} x_i = c(\mathcal{N})$ (\textit{i.e.,} a vector satisfying the \textit{efficiency} constraint). The set $\mathcal{N}$ is called the \textit{grand coalition}. This implies that the entirety of the cost of the grand coalition $c(\mathcal{N})$ is paid by all the members of $\mathcal{N}$. By abuse of notation, for any coalition $S \subseteq \mathcal{N}$, we may denote by $x(S)$ the cost $\sum_{i \in S}x_i$ paid by the members of $S$. The process of constructing such vectors with certain properties is a \textit{cost allocation problem} \citep{costallocp}. Generally, the properties that are expected from cost allocation vectors involve some sort of \textit{fairness} idea.

\paragraph*{Traveling Salesman Game} A T.U. game $(\mathcal{N}, c)$ is a \textit{Traveling Salesman Game} (TSG) if there exists a distance matrix $d = (d_{i,j})_{i,j \in \mathcal{N} \cup \{0\}}$ satisfying the triangle inequality, such that for every coalition $S \subseteq \mathcal{N}$, $c(S)$ is the cost of the minimum cost Hamiltonian circuit on the set $S \cup \{0\}$, with arc costs defined by matrix $d$ \citep{originaltsg}. In our context, node $0$ represents a client shared by a group of farmers, where the vehicle must bring all the production. We say that the distance matrix $d$ induces $(\mathcal{N}, c)$. Whenever the distance matrix inducing a TSG $(\mathcal{N}, c)$ is symmetric, we refer to the game as a Symmetric Traveling Salesman Game (STSG), and whenever it is not necessarily symmetric we refer to it simply as a Traveling Salesman Game (TSG). 

\paragraph*{Subadditive games} A game $(\mathcal{N}, c)$ is subadditive if all $S, S' \subseteq \mathcal{N}$ are such that $c(S) + c(S') \geq c(S \cup S')$. If the distance matrix inducing a TSG satisfies the triangle inequality then the corresponding game is subadditive \citep{tsg}. In this paper, we consider that the triangle inequality is always satisfied. Notice that subadditivity entails that it is less expensive to have the grand coalition to cooperate than to divide the players into multiple coalitions.

\paragraph*{Core of subadditive games} 
The core of a T.U. game $(\mathcal{N},c)$ is a set of outcomes $(x_i)_{i \in \mathcal{N}}$ such that no subset of $\mathcal{N}$ pays more than their actual cost. More formally:
{\small \begin{equation*}
    Core(\mathcal{N}, c) = \left\{{(x_i)_{i \in \mathcal{N}}\ \left|\ x(\mathcal{N}) = c(\mathcal{N}), \text{ and }  \forall S \subseteq \mathcal{N},\ c(S) \geq x(S) \text{, and } \forall i \in \mathcal{N},\ x_i \geq 0 \right.}\right\}
\end{equation*}
}
We say that an outcome of a game is \textit{stable} if it belongs to the core of the game. Indeed, if it were the case that a subset $S \subset \mathcal{N}$ was such that $x(S) > c(S)$, then $S$ could deviate from the grand coalition in order to manage their transportation by themselves, thus causing instability. We consider that a stable allocation is also \textit{fair}, although whether this is a philosophically correct term is a valid question that is out of the scope of this paper.

The core of a STSG can be empty whenever there are 6 or more players \citep{tsgcostallocgame} and can never be empty when there are 5 or less players \citep{tsg5p}. In the general (non symmetric) case, it can be empty when there are 4 or more players and cannot be empty when there are 3 or less \citep{originaltsg}. \cite{tsg} provide an example of a Euclidean distance TSG with an empty core along with an efficient method to compute an element of the core (or to decide whether it is empty), although it remains an NP-hard problem.

\subsection{Semicore allocations and their importance in the context of SFSC} \label{sec:semicore}


The concept of semicore of T.U. games was introduced by \sylvain{\cite{semicore}}. Given that, in order to model the core of a game, $2^n$ computations of the cost function are required, and that computing the cost function of one coalition can already be NP-hard, some form of relaxation of the core could be needed. We can, for instance, restrict ourselves to the stability constraint of some subset of all coalitions. The semicore is a set of allocation vectors satisfying the stability constraints on the sets of size $1$ and of size $n - 1$. 

More formally, we say that an outcome $(x_i)_{i \in \mathcal{N}}$ is \textit{individually rational} (or satisfies the individual rationality constraint) when $x_i \leq c(\{i\})$ for all $i \in \mathcal{N}$. The value $c(\{i\})$ is called the individual rationality of a player $i$. Farmers satisfying their individual rationality have an interest in participating in the coalition, since they pay a value that is less or equal to what they would pay by working alone. This corresponds to the stability of the coalitions of size~$1$. Notice that there always exists an allocation which satisfies the individual rationality for all players.

An outcome satisfies \textit{the marginal contribution constraint} if $x_i \geq c(\mathcal{N}) - c(\mathcal{N} \setminus \{i\})$ for all $i \in \mathcal{N}$. The value $c(\mathcal{N}) - c(\mathcal{N} \setminus \{i\})$ is the marginal contribution (or cost) of player~$i$. We can consider that the farmers $i \in \mathcal{N}$ satisfying their marginal contribution constraint are profitable for the coalition $\mathcal{N} \setminus \{i\}$, since they compensate their presence in the association. Notice that this is equivalent to the stability constraint of the coalition $\mathcal{N} \setminus \{i\}$, and thus that $\sum_{j \in \mathcal{N} \setminus \{i\}} x_j \leq c(\mathcal{N} \setminus \{i\})$.

The semicore of a game $(\mathcal{N}, c)$ is a set of outcomes $(x_i)_{i \in \mathcal{N}}$ that are individually rational and that satisfy the marginal contribution constraints:
\begin{multline*}
    Semicore(\mathcal{N}, c) = \left\{(x_i)_{i \in \mathcal{N}}\ \left|\ x(\mathcal{N}) = c(\mathcal{N}), \text{ and }  \forall i \in \mathcal{N}, x_i \leq c(\{i\})\right.\right.\\\left.\text{and } x(\mathcal{N} \setminus \{i\}) \leq c(\mathcal{N} \setminus \{i\})\right\}
\end{multline*}
Notice that $Core(\mathcal{N}, c)\subseteq Semicore(\mathcal{N}, c)$, since the semicore is constrained by a subset of the stability constraints defining the core.

Whenever the distance matrix is allowed to be asymmetric, there exist instances of TSGs with an empty semicore with 4 or more players, and with 6 or more players in the symmetric case \citep{semicoretsg}. Since the core of a STSG with 5 or less players cannot be empty \citep{tsg5p}, and since the core is included in the semicore, we have that the semicore of STSGs with 5 or less players is always non-empty. We reach through the same reasoning the conclusion that the semicore of TSGs with 3 or less players cannot be empty.

\paragraph*{The importance of the semicore}

\sylvain{\cite{semicorenucleolus}} state that the stability constraints defining the semicore are "the most essential". This was confirmed by our discussions with the experts in logistics of the association studied. Many of the farmers belonging to the association believed that they would be better off delivering their orders by themselves, which proves the need to satisfy their individual rationality. Some also held the belief that it was too expensive for the association to manage the pickup and delivery of orders of some individuals, reflecting that the need to have every member to satisfy the marginal contribution is part of their intuition. These intuitions justified the study of semicore allocations in our context.

\section{Approximately fair allocations}\label{sec:approxfair}

Whenever the core of a T.U. game is empty, we compare and discuss some alternatives from the literature which allow to find allocations that are "approximately" fair. Table~\ref{tab:approxfair} summarizes these various models.

\begin{table}[htbp]
    \centering
    {\scriptsize \begin{tabular}{|c|c|c|c|}
        \hline
        & Cost of stability & $\varepsilon$-core & $\alpha$-core\\
        \hline
        Objective function & min $\varepsilon$ & min $\varepsilon$ & min $\alpha$  \\
        Efficiency constraint & $x(\mathcal{N}) = c(\mathcal{N}) - \varepsilon$ &  $x(\mathcal{N}) = c(\mathcal{N})$ & $x(\mathcal{N}) \geq \frac{1}{\alpha}c(\mathcal{N})$\\
        Stability constraints & $x(S) \leq c(S)$ & $x(S) \leq c(S) + f(S)\varepsilon$ & $x(S)\leq c(S)$\\
        Domain & $\varepsilon \geq 0$ & $\varepsilon \geq 0$ & $\alpha \geq 1$\\
        \hline
        Notation of & $CoS(\mathcal{N}, c)$ & $wOeC(\mathcal{N}, c)$: $f(S) = |S|$,
        & $OaC(\mathcal{N}, c)$\\
        the optimal solution & & $sOeC(\mathcal{N}, c)$: $f(S) = 1$ &\\
        \hline
        References & \cite{costofstability} & \cite{originalepsiloncore} & \cite{alphacoretsg} \\
        \hline
    \end{tabular}}
    \caption{Summary of optimal approximately fair allocations in the literature. Stability constraints hold for all $S \subset \mathcal{N}$. $OeC$ stands for "Optimal epsilon Core" and the $w$ and the $s$ preceeding $OeC$ stand for "weak" and "strong" respectively. $OaC$ stands for "Optimal alpha Core". Remember that for all concepts, $x_i \geq 0$ for all $i \in \mathcal{N}$.}\label{tab:approxfair}
    
\end{table}

\paragraph*{Cost of stability \citep{costofstability}} \label{par:cos}

A commonly used alternative to approximate the core is by injecting money in order to attain stability in the grand coalition. In the case of the association of farmers that we have been studying, this alternative could be used by the local governments if they want to subsidize the association. From the standpoint of the government, the idea would be to provide the smallest amount of money that they could possibly provide in order to allow stability. The cost of stability is the minimum reduction to the cost of the grand coalition in order to obtain a stable allocation \citep{costofstability}. Given a T.U. game $(\mathcal{N}, c)$, the cost of stability of $(\mathcal{N}, c)$, noted $CoS(\mathcal{N}, c)$, corresponds to the value of the optimal solution of the following linear program:
\begin{equation*}
    \begin{array}{rll}
        \text{minimize} & \varepsilon\\
        \text{subject to} & \sum_{i \in \mathcal{N}} x_i = c(\mathcal{N}) - \varepsilon\\
        & \sum_{i \in S} x_i \leq c(S) & \forall S \subseteq \mathcal{N}\\
        & x_i \geq 0 & \forall i \in \mathcal{N}\\
        & \varepsilon \geq 0
    \end{array}
\end{equation*}
Note that if $\varepsilon = 0$ is a feasible solution, then $Core(\mathcal{N}, c) \neq \emptyset$.

\paragraph*{$\varepsilon$-core \citep{originalepsiloncore}}

Another possibility to find an approximately fair allocation whenever the core of a game is empty, is by allowing some instability, although minimizing it as much as possible. \sylvain{\cite{originalepsiloncore}} introduce the concept of $\varepsilon$-core, which is further expanded \sylvain{by \cite{epsiloncore}}, by listing multiple ways to extend the concept of $\varepsilon$-core that are found in the literature. Generally speaking, \sylvain{the} $\varepsilon$-core of a game $(\mathcal{N}, c)$ can be defined as:
\begin{multline*}
    \varepsilon \text{-} core(\mathcal{N}, c) = \left\{ (x_i)_{i \in \mathcal{N}} \left| \sum_{i \in \mathcal{N}} x_i = c(\mathcal{N}) \text{ and } \forall S \subset \mathcal{N},\ \sum_{i \in S} x_i \leq c(S) + f(S)\varepsilon\right.\right.\\\text{and } \forall i \in \mathcal{N},\ x_i \geq 0 \Biggl\}
\end{multline*}
\sylvain{When $f(S) = |S|$ (resp. $f(S) = 1$) for all $S \subseteq \mathcal{N}$, this set is called the "weak-$\varepsilon$-core" (resp. "strong-$\varepsilon$-core") by \cite{originalepsiloncore}}. We use the notation $sOeC(\mathcal{N}, c)$ and $wOeC(\mathcal{N}, c)$ to refer to the smallest value of $\varepsilon$ such that the strong-$\varepsilon$-core and the weak-$\varepsilon$-core are not empty respectively. \sylvain{\cite{epsiloncore}} also study the case where $f(S) = c(S)$. Proposition~\ref{prop:epsiloncos} establishes a link between the cost of stability and the weak-$\varepsilon$-core of a game.

\paragraph*{$\alpha$-core \citep{alphacoretsg}}

\sylvain{\cite{alphacoretsg}} define the $\alpha$-core, for $\alpha \geq 1$ and for a game $(\mathcal{N}, c)$:
\begin{multline*}
    \alpha \text{-} Core(\mathcal{N}, c) = \left\{(x_i)_{i \in \mathcal{N}}\ \left|\ x(\mathcal{N}) \geq \frac{1}{\alpha}c(\mathcal{N}) \text{ and } \forall S \subseteq \mathcal{N},\ x(S) \leq c(S)\right.\right.\\ \text{and } \forall i \in \mathcal{N},\ x_i \geq 0 \Biggl\}
\end{multline*}
where $\frac{1}{\alpha}$ is the fraction of the cost of the grand coalition that we want to make the players pay. If $\alpha \text{-} Core(\mathcal{N}, c) \neq \emptyset$, the game $(\mathcal{N}, c)$ is said to be $\alpha$-budget balanced.
Once again, if $Core(\mathcal{N}, c) = \emptyset$, we want to maximize the cost paid by the players to allow stability, and thus we want to find \sylvain{the smallest $\alpha \geq 1$ such that $\alpha\text{-}Core(\mathcal{N},c) \neq \emptyset$. Such value of $\alpha$ is denoted by $OaC(\mathcal{N}, c)$}. Notice that $OaC(\mathcal{N}, c) = 1$ if and only if $Core(\mathcal{N}, c) \neq \emptyset$. Moreover:

\begin{align*}
    CoS(\mathcal{N}, c) = c(\mathcal{N})\left(1 - \frac{1}{OaC(\mathcal{N}, c)}\right)
\end{align*}

\sylvain{\cite{alphacoretsg}} also prove that, in the case where $(\mathcal{N}, c)$ is a TSG, $OaC(\mathcal{N}, c) \leq \frac{3}{2}$, which implies that $CoS(\mathcal{N}, c) \leq \frac{1}{3}c(\mathcal{N})$. In Section~\ref{sec:generalmethodbound} we \sylvain{give} a general method to find this type of bound.

\begin{proposition}\label{prop:epsiloncos}
    Let $(\mathcal{N}, c)$ be a cost function game such that $Core(\mathcal{N}, c) = \emptyset$. We have: $CoS(\mathcal{N}, c) = n \cdot wOeC(\mathcal{N}, c)$.
    \begin{proof}
        Let:
        \begin{align*}
            & A^\varepsilon_{(\mathcal{N}, c)} = \left\{(x_i)_{i \in \mathcal{N}}\ \left|\ x(\mathcal{N}) = c(\mathcal{N}) - \varepsilon \text{ and } \forall S \subset \mathcal{N},\ x(S) \leq c(S) \text{ and } \forall i \in \mathcal{N}, x_i \geq 0 \right. \right\}\\
            & B^\varepsilon_{(\mathcal{N}, c)} = \text{weak-}\varepsilon\text{-}core(\mathcal{N}, c)
        \end{align*}
        and let $\varepsilon_a = CoS(\mathcal{N}, c) = min \{ \varepsilon\ |\ A^\varepsilon_{(\mathcal{N}, c)} \neq \emptyset\}$ and $\varepsilon_b = wOeC(\mathcal{N}, c) = min \{\varepsilon\ |\ B^\varepsilon_{(\mathcal{N}, c)} \neq \emptyset\}$.
        
        We first show that $n \varepsilon_b \leq \varepsilon_a$.
        Let $(x_i)_{i \in \mathcal{N}} \in A^{\varepsilon_a}_{(\mathcal{N}, c)}$, and $x' = (x'_i)_{i \in \mathcal{N}}$ such that $x'_i = x_i + \frac{\varepsilon_a}{n}$ for each $i \in \mathcal{N}$. We have:
        \begin{align*}
            \sum_{i \in \mathcal{N}} x'_i = \sum_{i \in \mathcal{N}} x_i + n\frac{\varepsilon_a}{n} = c(\mathcal{N}) - \varepsilon_a + \varepsilon_a = c(\mathcal{N}) 
        \end{align*}
        We also have that, for all $S \subset \mathcal{N}$:
        \begin{align*}
            \sum_{i \in S} x'_i = \sum_{i \in S} x_i + |S|\frac{\varepsilon_a}{n} \leq c(S) + |S|\frac{\varepsilon_a}{n}
        \end{align*}
        Therefore $B^{\frac{\varepsilon_a}{n}}_{(\mathcal{N}, c)} \neq \emptyset$, implying that $\varepsilon_b \leq \frac{\varepsilon_a}{n}$. Hence $n \varepsilon_b \leq \varepsilon_a$.
        
         We now show that $n \varepsilon_b \geq \varepsilon_a$. Let $x = (x_i)_{i \in \mathcal{N}} \in B^{\varepsilon_b}_{(\mathcal{N}, c)}$ and $x' = (x'_i)_{i \in \mathcal{N}}$ be such that $x'_i = x_i - \varepsilon_b$ for each $i \in \mathcal{N}$. We have:
        \begin{align*}
            \sum_{i \in \mathcal{N}} x'_i = \sum_{i \in \mathcal{N}} x_i - n \varepsilon_b = c(\mathcal{N}) - n \varepsilon_b
        \end{align*}
        Besides, for all $S \subset \mathcal{N}$:
        \begin{align*}
            & \sum_{i \in S} x'_i = \sum_{i \in S} x_i - |S|\varepsilon_b \leq c(S) + |S|\varepsilon_b - |S|\varepsilon_b = c(S)\\
            \Rightarrow & \sum_{i \in \mathcal{N}} x'_i \leq c(S)
        \end{align*}
        We therefore have $A^{n\varepsilon_b}_{(\mathcal{N}, c)} \neq \emptyset$, which implies that $\varepsilon_a \leq n \varepsilon_b$. We conclude $\varepsilon_a = n \varepsilon_b$.
    \end{proof}
\end{proposition}

To summarize it all, Table~\ref{tab:approxfair} shows side by side the different models for approximately fair allocations. Notice that they always admit a feasible solution, hence these notions are well defined.

The following equation summarizes the links between them:
\begin{align*}
    CoS(\mathcal{N}, c) = c(\mathcal{N})\left(1 - \frac{1}{OaC(\mathcal{N}, c)}\right) = n \cdot wOeC(\mathcal{N}, c)
\end{align*}

\section{Computation of Approximate Semicore Stability}\label{sec:approxsemicore}

We now turn to the concept of semicore. As stated in Section~\ref{sec:semicore}, the semicore of a TSG can be empty. In these situations, we could either find another solution concept, or study relaxations of the semicore. In this paper we are interested in the latter. Similarly to how the cost of stability is defined in Section~\ref{sec:approxfair}, we can define the \textit{cost of semicore stability} as the minimum reduction to the cost of the grand coalition in order to obtain a game with a non-empty semicore. We denote the cost of semicore stability \sylvain{by} $CoSS(\mathcal{N}, c)$. Table~\ref{tab:approxfair} can also be used to summarize all the remaining concepts of approximate semicore stability; it suffices to apply the stability constraints to all subsets $S \subset \mathcal{N}$ such that $|S| = 1$ (individual rationality) or $|S| = n - 1$ (marginal contribution) instead of to all subsets of $\mathcal{N}$. The equivalent of the $\varepsilon$-core, in the case of the semicore, is referred to as $\varepsilon$-semicore, and the equivalent of the $\alpha$-core is the $\alpha$-semicore. We can also make the distinction between the strong and the weak-$\varepsilon$-semicore depending on the definition of function $f$ (\textit{i.e.} $f(S) = 1$ and $f(S) = |S|$ respectively). We use the notations $wOeS(\mathcal{N}, c)$ and $sOeS(\mathcal{N}, c)$ to represent the value of the weak and strong "Optimal epsilon Semicore" of the game $(\mathcal{N}, c)$. In the remaining part of this paper we mainly focus on the strong version of the $\varepsilon$-semicore, and thus $f(S) = 1$ for all $S \subset \mathcal{N}$.


In this section, for subadditive games, we give exact formulas for the cost of semicore stability $CoSS(\mathcal{N}, c)$ in Theorem~\ref{theorem:coss} and for the optimal strong-$\varepsilon$-semicore $sOeS(\mathcal{N}, c)$ in Theorem~\ref{theorem:epsilonsemicore}. \sylvain{Let us start with a lemma that will be useful to prove Theorems~\ref{theorem:coss} and \ref{theorem:epsilonsemicore}.}

\begin{lemma}\label{lemma:lemmaepsilon}
    Let $(\mathcal{N}, c)$ be a subadditive T.U. game. Let $c(\mathcal{N}) \geq \varepsilon \geq 0$ and let $(\mathcal{N}, c^\varepsilon)$ be a T.U. game such that $c^\varepsilon(\mathcal{N}) = c(\mathcal{N}) - \varepsilon$ and $c^\varepsilon(S) = c(S)$ for $S \subset \mathcal{N}$. The game $(\mathcal{N}, c^\varepsilon)$ is subadditive.
\end{lemma}
\begin{proof}
    The case $S, S' \subset \mathcal{N}$, with $S \cup S' \subset \mathcal{N}$, derives directly from the subadditivity of $(\mathcal{N}, c)$, since the cost function does not change for those subsets. If $S, S' \subset \mathcal{N}$ and $S \cup S' = \mathcal{N}$, we have $c^\varepsilon(S) + c^\varepsilon(S') = c(S) + c(S') \geq c(\mathcal{N}) \geq c(\mathcal{N}) - \varepsilon = c^\varepsilon(\mathcal{N})$. Moreover, for any $S' \subseteq \mathcal{N}$, we have that $c^\varepsilon(\mathcal{N}) + c^\varepsilon(S') \geq c^\varepsilon(\mathcal{N})$, since $c^\varepsilon(S') = c(S')$ is positive for every $S' \subseteq \mathcal{N}$.
\end{proof}

\begin{theorem}\label{theorem:coss}
    Let $(\mathcal{N}, c)$ be a subadditive game such that $Semicore(\mathcal{N}, c) = \emptyset$. We have: $$CoSS(\mathcal{N}, c) = c(\mathcal{N}) - \frac{1}{n - 1} \sum_{j \in \mathcal{N}} c(\mathcal{N} \setminus \{j\})$$
\end{theorem}

\begin{proof}
    Let $(\mathcal{N}, c)$ be a subadditive game with an empty semicore. Let $\varepsilon \geq 0$ and let $(\mathcal{N}, c^\varepsilon)$ be a T.U. game such that $c^\varepsilon(\mathcal{N}) = c(\mathcal{N}) - \varepsilon$ and $c^\varepsilon(S) = c(S)$ for $S \subset \mathcal{N}$. We have that $CoSS(\mathcal{N}, c) = \min \{\varepsilon~\geq~0\ |\ Semicore(\mathcal{N}, c^\varepsilon) \neq \emptyset\}$. Let $\varepsilon^* = CoSS(\mathcal{N}, c)$. Lemma~\ref{lemma:lemmaepsilon} proves that $(\mathcal{N}, c^{\varepsilon^*})$ is subadditive. \sylvain{\cite{semicoretsg} prove} that for any subadditive game $(\mathcal{N}, c)$, we have $Semicore(\mathcal{N}, c) = \emptyset \iff \sum_{j \in \mathcal{N}} \left( c(\mathcal{N}) - c(\mathcal{N} \setminus \{j\}) \right) > c(\mathcal{N})$. 
    
    Since $(\mathcal{N}, c^{\varepsilon^*})$ is subadditive and has a non empty semicore, it entails that $\sum_{j \in \mathcal{N}} (c^{\varepsilon^*}(\mathcal{N}) - c^{\varepsilon^*}(\mathcal{N} \setminus \{j\})) \leq c^{\varepsilon^*}(\mathcal{N})$. Hence, we have:
    \begin{align*}
        & \quad \sum_{j \in \mathcal{N}} \left(c(\mathcal{N}) - \varepsilon^* - c(\mathcal{N} \setminus \{j\})\right) \leq c(\mathcal{N}) - \varepsilon^*\\
        \Rightarrow & \quad c(\mathcal{N}) - \frac{1}{n - 1} \sum_{j \in \mathcal{N}} c(\mathcal{N} \setminus \{j\}) \leq \varepsilon^*
    \end{align*}
    Let $\varepsilon = c(\mathcal{N}) - \frac{1}{n - 1}\sum_{j \in \mathcal{N}} c(\mathcal{N} \setminus \{j\})$. We proceed to prove that $Semicore(\mathcal{N}, c^\varepsilon) \neq \emptyset$ and thus $\varepsilon^* = \varepsilon$. We have:\\
    $\sum_{j \in \mathcal{N}} \left((c(\mathcal{N}) - \varepsilon) - c(\mathcal{N} \setminus \{j\}) \right)$
    {\allowdisplaybreaks
    \begin{align*}
        & & = & \quad n c(\mathcal{N}) - \left(\sum_{j \in \mathcal{N}} c(\mathcal{N} \setminus \{j\}\right) - n\varepsilon\\
        & \quad & = & \quad nc(\mathcal{N}) - \left(\sum_{j \in \mathcal{N}} c(\mathcal{N} \setminus \{j\}\right) - n\left(c(\mathcal{N}) - \frac{1}{n - 1}\sum_{j \in \mathcal{N}} c(\mathcal{N} \setminus \{j\})\right)\\
        & \quad & = & \quad \frac{1}{n - 1}\sum_{j \in \mathcal{N}} c(\mathcal{N} \setminus \{j\})\\
        & \quad & = & \quad c(\mathcal{N}) - \varepsilon\\
        & \quad & = & \quad c^\varepsilon(\mathcal{N})
    \end{align*}
    }
    Hence, $\sum_{j \in \mathcal{N}} c^\varepsilon(\mathcal{N}) - c^\varepsilon(\mathcal{N} \setminus \{j\}) = c^\varepsilon(\mathcal{N})$, which in turn implies that $Semicore(\mathcal{N}, c^\varepsilon) \neq \emptyset$, and thus $\varepsilon = \varepsilon^*$.
\end{proof}

\begin{theorem}\label{theorem:epsilonsemicore}
    Let $(\mathcal{N}, c)$ be a subadditive game such that $Semicore(\mathcal{N}, c) = \emptyset$. We have: $$sOeS(\mathcal{N}, c) = \frac{(n - 1)}{n}c(\mathcal{N}) - \frac{1}{n}\sum_{i \in \mathcal{N}} c(\mathcal{N} \setminus \{i\})$$
\end{theorem}
\begin{proof}
    Let $(\mathcal{N}, c)$ be a subadditive T.U. game such that $Semicore(\mathcal{N}, c) = \emptyset$ and $(\mathcal{N}, c^\varepsilon$) be a game such that $c^\varepsilon(\mathcal{N}) = c(\mathcal{N})$, and $c^\varepsilon(S) = c(S) + \varepsilon$ for $S \subset \mathcal{N}$ and $\varepsilon \in \mathbb{R}_+$. The game $(\mathcal{N}, c^\varepsilon)$ is subadditive (the proof is similar to that of Lemma~\ref{lemma:lemmaepsilon}), and thus, for all $\varepsilon \geq 0$, $Semicore(\mathcal{N}, c^\varepsilon) \neq \emptyset \iff \sum_{j \in \mathcal{N}} \left( c^\varepsilon(\mathcal{N}) - c^\varepsilon(\mathcal{N} \setminus \{j\}) \right) \leq c^\varepsilon(\mathcal{N})$. Let $\varepsilon^* = \min \{ \varepsilon\ |\ Semicore(\mathcal{N}, c^\varepsilon) \neq \emptyset \} = sOeS(\mathcal{N}, c)$. We have:
    {\allowdisplaybreaks
    \begin{align*}
        & \sum_{j \in \mathcal{N}} c^{\varepsilon^*}(\mathcal{N}) - c^{\varepsilon^*}(\mathcal{N} \setminus \{j\}) \leq c^{\varepsilon^*}(\mathcal{N})\\
        \Rightarrow & \sum_{j \in \mathcal{N}} \left( c(\mathcal{N}) - c(\mathcal{N} \setminus \{j\})\right) - n\varepsilon^*  \leq  c(\mathcal{N})\\
        \Rightarrow & \frac{n-1}{n}c(\mathcal{N}) - \frac{1}{n}\sum_{j \in \mathcal{N}} c(\mathcal{N} \setminus \{j\}) \leq \varepsilon^*
    \end{align*}
    }
    Let $\varepsilon = \frac{n-1}{n}c(\mathcal{N}) - \frac{1}{n}\sum_{j \in \mathcal{N}} c(\mathcal{N} \setminus \{j\})$. We proceed to prove $Semicore(\mathcal{N}, c^\varepsilon) \neq \emptyset$.
    {\allowdisplaybreaks
    \begin{align*}
        & \sum_{j \in \mathcal{N}} \left(c^{\varepsilon}(\mathcal{N}) - c^\varepsilon(\mathcal{N} \setminus \{j\})\right) & = & \sum_{j \in \mathcal{N}} \left( c(\mathcal{N}) - c(\mathcal{N} \setminus \{j\})\right) - n\varepsilon\\
        & & = & \sum_{j \in \mathcal{N}} \left( c(\mathcal{N}) - c(\mathcal{N} \setminus \{j\})\right) - n\left( \frac{n-1}{n}c(\mathcal{N}) - \frac{1}{n}\sum_{j \in \mathcal{N}} c(\mathcal{N} \setminus \{j\}) \right)\\
        & & = & c(\mathcal{N}) = c^\varepsilon(\mathcal{N})
    \end{align*}
    }
    Therefore $Semicore(\mathcal{N}, c^\varepsilon) \neq \emptyset$, hence $\varepsilon^* = \varepsilon$.
\end{proof}

We deduce from Theorems~\ref{theorem:coss} and~\ref{theorem:epsilonsemicore} the following property:
\begin{align*}
    CoSS(\mathcal{N}, c) = \frac{n}{n-1} \cdot sOeS(\mathcal{N}, c)
\end{align*}
Hence, the larger the number of players in the game, the closer the cost of semicore stability is to the optimal $\varepsilon$-semicore.

\section{Bounds of the cost of (semi)core stability}\label{sec:methodboundcoss}

In this section we focus on bounds of the cost of core and semicore stability. Recall that the cost of (semicore) stability is the smallest amount of money one needs to subsidize the grand coalition with to allow stability. We first generalize a method used in the literature to bound the cost of stability in the particular case of the TSG (Section~\ref{sec:generalmethodbound}) and we extend this method to bound the cost of semicore stability. Given that the decision problem of the TSP remains an NP-complete problem, we aim to bound the cost of semicore stability of the TSP using a polynomial formula (Section~\ref{sec:polybounds}). 

\subsection{A general method}\label{sec:generalmethodbound}

We can construct a method to find bounds for the cost of stability of a game $(\mathcal{N}, c)$ as follows:
\begin{enumerate}[(a)]
    \item Propose a cost function game $(\mathcal{N}, \Bar{c})$.
    \item Prove that $\Bar{c}(S) \leq c(S)$ for all $S \subseteq \mathcal{N}$.
    \item Prove that $Core(\mathcal{N}, \Bar{c}) \neq \emptyset$.
    \item Bound the ratio $\frac{c(\mathcal{N})}{\Bar{c}(\mathcal{N})}$, \textit{i.e.,} $\frac{c(\mathcal{N})}{\Bar{c}(\mathcal{N})} \leq \alpha$, or, equivalently, bound the difference $c(\mathcal{N}) - \Bar{c}(\mathcal{N}) \leq c(\mathcal{N}) - \frac{1}{\alpha}c(\mathcal{N})$, with $\alpha \geq 1$.
    \item Conclude $CoS(\mathcal{N}, c) \leq c(\mathcal{N}) - \frac{1}{\alpha}c(\mathcal{N})$.
\end{enumerate}
We reach the conclusion (e) whenever premises (b), (c) and (d) hold for some game $(\mathcal{N}, \bar{c})$, by the fact that an allocation $x = (x_i)_{i \in \mathcal{N}} \in Core(\mathcal{N}, \Bar{c})$ (which must exist by (c)) is such that $x(S) \leq \Bar{c}(S) \leq c(S)$ for all $S \subseteq \mathcal{N}$ (by (b)), and $x(\mathcal{N}) = \Bar{c}(\mathcal{N}) = c(\mathcal{N}) - (c(\mathcal{N}) - \Bar{c}(\mathcal{N}))$. Hence, if we set $\varepsilon = c(\mathcal{N}) - \Bar{c}(\mathcal{N})$, along with vector $x$, we obtain a feasible solution in the polyhedron constraining the cost of stability of game $(\mathcal{N}, c)$, given in Table~\ref{tab:approxfair}. Thus $CoS(\mathcal{N}, c) \leq \varepsilon = c(\mathcal{N}) - \Bar{c}(\mathcal{N})$ which implies that $CoS(\mathcal{N}, c) \leq c(\mathcal{N}) - \frac{1}{\alpha}c(\mathcal{N})$ (by (d)).

A more precise example of this method is the following \citep{alphacoretsg}: by using the linear relaxation of the TSP to compute the game $(\mathcal{N}, \Bar{c})$, proving this game has a non-empty core, and proving that the integrality gap is upper bounded  by $\frac{3}{2}$, one obtains that the cost of stability of a TSG $(\mathcal{N}, c)$ is \sylvain{upper} bounded by $\frac{1}{3}c(\mathcal{N})$. Theorem \ref{theorem:boundcosst} illustrates this procedure by proving a weaker bound than the one \sylvain{provided by \cite{alphacoretsg}}, although it serves the purpose of showing another example of how the method can be used. The proof of Theorem~\ref{theorem:boundcosst} relies on the Minimum Cost Spanning Tree (MCST) game, where the objective is to allocate the cost of a minimum cost spanning tree with one depot and $n$ locations to the $n$ locations \citep{mcstg}. If the conjectured Held-Karp bound $\frac{4}{3}$ for the TSP \citep{heldkarpbound} were to be proven, this same method would lead to bound the cost of stability of TSGs by $\frac{1}{4}c(\mathcal{N})$.
\begin{theorem}\label{theorem:boundcosst}
        Let $(\mathcal{N}, c)$ be a Traveling Salesman Game induced by a distance matrix $d = (d_{i,j})_{i, j \in \mathcal{N}}$. We have: $CoS(\mathcal{N}, c) \leq \frac{1}{2}c(\mathcal{N})$. In other words, the cost of stability of a Traveling Salesman Game is at most half the cost of the grand coalition.
\end{theorem}
\begin{proof}
This proof follows the structure of the method provided at the beginning of Section~\ref{sec:generalmethodbound}.
\begin{enumerate}[(a)]
    \item Let $(\mathcal{N}, c^{st})$ be a Minimum Cost Spanning Tree Game (MCSTG) induced by distance matrix $d$.
    \item When removing an edge from a Hamiltonian tour, we obtain a spanning tree. This spanning tree has a cost larger than or equal to a minimum cost spanning tree. Hence we deduce that $c^{st}(S) \leq c(S)$ for all $S \subseteq \mathcal{N}$ (since $c(S)$ is the cost of a Hamiltonian tour on locations $S \cup \{0\}$  and $c^{st}(S)$ is the cost of a minimum spanning tree on the same set of locations).
    \item It is known \citep{mcstg} that every 
    MCSTG $(\mathcal{N}, c')$ is such that $Core(\mathcal{N}, c') \neq \emptyset$.
    \item Consider $\mathcal{T} = (v_0, v_1, ..., v_k, v_0)$, a Eulerian tour on a graph obtained by duplicating all edges of a minimum cost spanning tree on vertices $\mathcal{N} \cup \{0\}$. We denote \sylvain{by} $cost_d(H)$ the cost of some graph $H$ weighted by distance matrix $d$. We have $cost_d(\mathcal{T}) = 2 \cdot c^{st}(\mathcal{N})$. Let $\mathcal{C}$ be a Hamiltonian tour obtained by iterating over $\mathcal{T}$, by only keeping each vertex in the order of their first appearance in $\mathcal{T}$. As in the classical 2-approximation of the TSP \citep{tsp2approx}, one can prove that $cost_d(\mathcal{C}) \leq cost_d(\mathcal{T})$ (because of the triangle inequality). $\mathcal{C}$, being a Hamiltonian tour, has a cost that is larger than or equal to a minimum cost Hamiltonian tour. Hence $c(\mathcal{N}) \leq cost_d(\mathcal{C})$. Hence $c(\mathcal{N}) \leq cost_d(\mathcal{C}) \leq cost_d(\mathcal{T}) = 2 \cdot c^{st}(\mathcal{N})$. We conclude that $c(\mathcal{N}) \leq 2 \cdot c^{st}(\mathcal{N})$
    \item Let $x = (x_i)_{i \in \mathcal{N}} \in Core(\mathcal{N}, c^{st})$. From the previous result, we deduce that $\frac{1}{2}c(\mathcal{N}) \leq \sum_{i \in \mathcal{N}} x_i = c^{st}(\mathcal{N}) \leq c(\mathcal{N})$. Since $x \in Core(\mathcal{N}, c^{st})$, we have that $\sum_{i \in S} x_i \leq c^{st}(S)$ for every $S \subseteq \mathcal{N}$, and as we stated before, $c^{st}(S) \leq c(S)$ for every $S \subseteq \mathcal{N}$, hence $\sum_{i \in S} x_i \leq c(S)$ for every $S \subseteq \mathcal{N}$. Then, if we set $\varepsilon = c(\mathcal{N}) - c^{st}(\mathcal{N})$, along with $x \in Core(\mathcal{N}, c^{st})$, we obtain a feasible solution in the polyhedron representing the cost of stability.
    Finally, $c^{st}(\mathcal{N}) \geq \frac{1}{2} c(\mathcal{N}) \Rightarrow c(\mathcal{N}) - c^{st}(\mathcal{N}) \leq c(\mathcal{N}) - \frac{1}{2}c(\mathcal{N}) \Rightarrow \varepsilon \leq \frac{1}{2}c(\mathcal{N})$. 
    By the fact that $CoS(\mathcal{N}, c) \leq \varepsilon$, we conclude $CoS(\mathcal{N}, c) \leq \frac{1}{2}c(\mathcal{N})$.        
\end{enumerate}
\end{proof}
For the semicore stability, since $CoSS(\mathcal{N}, c) \leq CoS(\mathcal{N}, c)$, any upper bound of the cost of stability is also an upper bound of the cost of semicore stability. Besides, one could adapt the method as follows:
\begin{enumerate}[(a)]
    \item Propose a cost function game $(\mathcal{N}, \Bar{c})$.
    \item Prove that $\Bar{c}(S) \leq c(S)$ for all $S \subseteq \mathcal{N}$.
    \item Prove that there exists an allocation vector of $(\mathcal{N}, \Bar{c})$ (\textit{i.e.,} a vector $(x_i)_{i \in \mathcal{N}}$ such that $\sum_{i \in \mathcal{N}} x_i = \Bar{c}(\mathcal{N})$), which satisfies the individual rationality and the marginal contribution constraints of the game $(\mathcal{N}, c)$.
    \item Bound the difference $c - \Bar{c} \leq \beta(\Bar{c}, c, \mathcal{N})$.
    \item Conclude that $CoSS(\mathcal{N}, c) \leq \beta(\Bar{c}, c, \mathcal{N})$.
\end{enumerate}
We can prove that this method works using the same arguments as the ones we used before for the cost of stability. Theorem~\ref{theorem:cossbound2} illustrates this method for the class of subadditive games (more general than TSGs). However, instead of finding a precise cost function game, we simply find an allocation vector $(x_i)_{i \in \mathcal{N}}$ such that $\sum_{i \in \mathcal{N}} x_i = c(\mathcal{N} \setminus \{M\})$ for some $M \in \mathcal{N}$ and prove it satisfies the individual rationality and marginal contribution constraints on $(\mathcal{N}, c)$.
\begin{theorem}\label{theorem:cossbound2}
    Let $(\mathcal{N}, c)$ be a subadditive game. If $Semicore(\mathcal{N}, c) = \emptyset$, then:\\ $CoSS(\mathcal{N}, c) \leq \max_{j \in \mathcal{N}} (c(\mathcal{N}) - c(\mathcal{N} \setminus \{j\}))$. In other words, the cost of semicore stability is smaller or equal than the maximum marginal cost.
\end{theorem}
\begin{proof}
    As a reminder, the cost of semicore stability of a game $(\mathcal{N}, c)$ is the optimal solution to the following MIP:
    {\allowdisplaybreaks
    \begin{align}
        \min \quad \varepsilon\\
        \text{ \textbf{s.t.} } & \sum_{i \in \mathcal{N}} x_i = c(\mathcal{N}) - \varepsilon\\
        & \sum_{j \in \mathcal{N} \setminus \{i\}} x_j \leq c(\mathcal{N} \setminus \{i\}) & \forall i \in \mathcal{N}\\
        & x_i \leq c(\{i\}) & \forall i \in \mathcal{N}\\
        & x_i \geq 0 & \forall i \in \mathcal{N}
    \end{align}
    }
    This proof follows the structure of the proposed method. 
    \begin{enumerate}[(a)]
        \item Let $(\mathcal{N}, c)$ be a game with an empty semicore, let $M \in \mathcal{N}$ be the player with maximum marginal cost, and $(\mathcal{N}, \Bar{c})$ be a game such that $\bar{c}(\mathcal{N}) = c(\mathcal{N} \setminus \{M\})$ and $\Bar{c}(S) = c(S)$ for all $S \subset \mathcal{N}$.
        \item Notice that $\Bar{c}(S) \leq c(S)$ for all $S \subseteq \mathcal{N}$ (we only decreased the value of the grand coalition)
        \item Let $\varepsilon = \max_{j \in \mathcal{N}} c(\mathcal{N}) - c(\mathcal{N} \setminus \{j\}) = c(\mathcal{N}) - c(\mathcal{N} \setminus \{M\})$. Since $\sum_{j \in \mathcal{N}} c(\{j\}) \geq c(\mathcal{N}) \geq c(\mathcal{N} \setminus \{M\})$ (by subadditivity and the fact that costs are positive), then the allocation vector $(x_i)_{i \in \mathcal{N}} = \left(\frac{c(\mathcal{N} \setminus \{M\})}{\sum_{j \in \mathcal{N}} c(\{j\})}c(\{i\})\right)_{i \in \mathcal{N}}$ satisfies the individual rationality constraints of the game $(\mathcal{N}, c)$.  We also have the efficiency constraint:
        \begin{multline}
            \sum_{i \in \mathcal{N}} \frac{c(\mathcal{N} \setminus \{M\})}{\sum_{j \in \mathcal{N}} c(\{j\})}c(\{i\}) = \frac{c(\mathcal{N} \setminus \{M\})}{\sum_{j \in \mathcal{N}} c(\{j\})} \sum_{i \in \mathcal{N}} c(\{i\})\\ = c(\mathcal{N} \setminus \{M\}) = c(\mathcal{N}) - \varepsilon
        \end{multline}
        Moreover, for all $i \in \mathcal{N}$, $\sum_{j \in \mathcal{N} \setminus \{i\}} x_j \leq \sum_{j \in \mathcal{N}} x_j = c(\mathcal{N} \setminus \{M\}) \leq c(\mathcal{N} \setminus \{i\})$. Hence $x$ satisfies the marginal contribution constraints of game $(\mathcal{N}, c)$. 
        \item One has directly $c(\mathcal{N}) - \Bar{c}(\mathcal{N}) = c(\mathcal{N}) - c(\mathcal{N} \setminus \{M\})$.
        \item Since $\varepsilon = c(\mathcal{N}) - c(\mathcal{N} \setminus \{M\})$, along with vector $x$, provide a feasible solution in the polyhedron representing the cost of semicore stability, we can conclude $CoSS(\mathcal{N}) \leq \varepsilon = c(\mathcal{N}) - c(\mathcal{N} \setminus \{M\})$.
    \end{enumerate}
\end{proof}

\subsection{Polynomial bounds for the cost of semicore stability of TSGs}\label{sec:polybounds}

The formula for the exact value of the cost of stability given in Theorem \ref{theorem:coss} requires to compute, in the case of TSGs, the optimal tour for every coalition of size $n - 1$ besides the optimal tour for the grand coalition. This corresponds to $n$ computations of TSPs on $n$ locations, and one on $n + 1$ locations (a NP-hard problem). We provide in Proposition \ref{prop:bound} an easy to understand and polynomially computable bound for the cost of semicore stability, which allows whoever wants to subsidize the association to easily estimate how much value they would need to provide in order to create semicore stability. 
\begin{proposition}\label{prop:bound}
    Let $(\mathcal{N}, c)$ be a TSG such that $Semicore(\mathcal{N}, c) = \emptyset$. Let $M \in \mathcal{N}$ be such that $c(\{M\}) \geq c(\{i\})$ for all $i \in \mathcal{N}$. We have that $CoSS(\mathcal{N}, c) \leq \frac{1}{n - 1}\sum_{j \in \mathcal{N} \setminus \{M\}} (d_{0,j} + d_{j,0})$. In other words, the cost of stability does not \sylvain{exceed} the average of the $n - 1$ smallest individual rationalities.
\end{proposition}
\begin{proof}
Let $(\mathcal{N}, c)$ be a TSG such that $Semicore(\mathcal{N}, c) = \emptyset$ and $M \in \mathcal{N}$ the player with the largest individual rationality. By Theorem \ref{theorem:coss}, we have:\\
$CoSS(\mathcal{N}, c) \quad$
{\allowdisplaybreaks
\begin{align*}
     = \quad & c(\mathcal{N}) - \frac{1}{n - 1}\sum_{j \in \mathcal{N}} c(\mathcal{N} \setminus \{j\}) \\
     \quad = \quad & \frac{(n-1)c(\mathcal{N})}{n-1} - \frac{1}{n - 1} \sum_{j \in \mathcal{N}} c(\mathcal{N} \setminus \{j\})\\
     \quad \leq \quad & \frac{1}{n-1}\left(\sum_{j \in \mathcal{N} \setminus \{M\}} \left( c(\mathcal{N} \setminus  \{j\}) + c(\{j\})\right) - \sum_{j \in \mathcal{N} \setminus \{M\}} c(\mathcal{N} \setminus \{j\})\right) & \text{(by subadditivity)}\\
     \quad = \quad & \frac{1}{n - 1}\left(\sum_{j \in \mathcal{N} \setminus \{M\}} c(\{j\})\right)\\
     \quad = \quad & \frac{1}{n - 1}\left(\sum_{j \in \mathcal{N} \setminus \{M\}} \left( d_{0, j} + d_{j, 0} \right) \right)
\end{align*}
}
\end{proof}
\section{Conclusion}

In this paper we investigated approximate stability concepts on subadditive T.U. games, by first identifying, in the literature, methods to approximate core allocations whenever the core is empty, and then establishing the link between such methods. Through a short presentation of the context of application of this research, we justified the use of semicore allocations. We stated that, in a context of collaborative transportation in a Short Food Supply Chain where a group of farmers shares a vehicle and a client, and where the cost of a TSP must be allocated to the farmers, reaching a stable cost allocation (\textit{i.e.} where no subset of farmers wishes to deviate) might become impractical. The semicore, being a relaxation of the core where only the stability constraints of coalitions of size $1$ and $n - 1$ are taken into account, was the solution concept we chose for this study. It has been proven in the literature that the semicore of TSGs can be empty and thus we proposed to extend all the methods of approximate stability to the semicore. We properly defined the cost of semicore stability and the $\varepsilon$-semicore of subadditive T.U. games, and provided an exact formula to compute their respective optimums. We identified in the literature a method to bound the cost of (core) stability. A generalization of the method, along with its translation to the concept of cost of semicore stability was provided. We applied the general method to TSGs and then provided polynomial and simple bounds for the cost of semicore stability of TSGs.


\paragraph*{Future research}
The bounds for the cost of (semicore) stability mentioned in this paper can be improved. We mentioned for instance the Held-Karp bound, which if it happened to be proven to be $\frac{4}{3}$ would allow us to bound the cost of stability by $\frac{1}{4}$ of the cost of the grand coalition. Besides, given that the semicore is less constraining than the core, we could explore other ways to better bound the cost of semicore stability.\\ 
In practice, the results presented in this paper were used as a discussion tool between the association and the local government. However, the transportation problem faced by the association at each delivery day is not a simple TSP. In reality, there are multiple clients, all orders must go first through a depot in order to condition them, the vehicle does not have an unlimited capacity and there are time windows that must be taken into account. The complexity also rises from the fact that, if we wanted to compute a stable allocation, a precise cost function would be required. This is a problem first by the fact that it has a worse time complexity than the TSP (knowing the decision problem of the TSP is NP-hard). Secondly, it is possible that some subset of the farmers find the fact of not using a depot more practical, less time consuming and less expensive, and thus the estimated cost $c(S)$ of such subset $S \subseteq \mathcal{N}$ changes its definition depending on whether the subset wants to use a depot or not. If it is not the case, we would be facing a Pickup and Delivery problem with Time Windows. This makes the real life situation more complex. Can the same method used to bound the cost of (semicore) stability of TSGs in this paper be used in this real life problem?
\section*{Acknowledgement}

This work has been partially supported by the LabEx PERSYVAL-Lab (ANR-11-LABX-0025-01) funded by the French program Investissement d’avenir.

\bibliography{cas-refs}

\end{document}